\documentclass[reqno]{amsart}
\usepackage{cite}
\usepackage{amsmath}
\usepackage{amsfonts}
\usepackage{amssymb}
\usepackage{amsthm}
\usepackage{xcolor}
\usepackage{geometry}
\usepackage{comment}

\usepackage{esint}

\newtheorem{theorem}{Theorem}
\newtheorem{proposition}{Proposition}
\newtheorem{lemma}{Lemma}
\newtheorem{corollary}{Corollary}
\newtheorem*{claim}{Claim}

\theoremstyle{definition}
\newtheorem{definition}{Definition}

\theoremstyle{remark}
\newtheorem{remark}{Remark}

\newcommand{\R}{\mathbb R}

\newcommand{\diam}{\operatorname{diam}}
\newcommand{\wt}[1]{{\widetilde{#1}}}

\title[Propagation of smallness for gradients of harmonic functions]{Propagation of smallness near codimension two for gradients of harmonic functions}
\author{Benjamin Foster and Josep M. Gallegos}

\begin{document}
	
	\begin{abstract}
		Let $u$ be a harmonic function in the unit ball $B_1 \subset \mathbb R^n$ satisfying $\sup_{B_1}|\nabla u|=1$. We show that if {$|\nabla u|$} is $\epsilon$-small on a set {$E\subset B_{1/2}$} with positive $(n-2+\delta)$-dimensional Hausdorff content for some $\delta>0$, then $\sup_{B_{1/2}} |\nabla u| \leq C \epsilon^\alpha$ with $C,\alpha>0$ depending only on $n,\delta$ and the $(n-2+\delta)$-Hausdorff content of $E$. This is an improvement over a similar result in \cite{LM} that required $\delta>1-c_n$ for a small dimensional constant $c_n$ and reaches the sharp threshold for the dimension of the smallness sets from which propagation of smallness can occur.
	\end{abstract}
	
	\maketitle
	\section{Introduction}
	
	One classical inequality for harmonic functions is the three spheres theorem, which states that if $u$ is a harmonic function and $0<r_1<r_2<r_3$, then
	\[
	\Vert u\Vert_{L^2(\partial B_{r_2})} \le \Vert u \Vert_{L^2(\partial B_{r_1})}^{\alpha} \Vert u \Vert_{L^2(\partial B_{r_3})}^{1-\alpha},
	\]
	where $0<\alpha<1$ depends on $r_1,r_2,r_3$. In particular, if we think of normalizing the norm on the largest sphere to be 1, then the smallness of $u$ on $\partial B_{r_1}$ propagates to the larger sphere $\partial B_{r_2}$.
	We can consider more general estimates of the form
	\begin{equation}\label{eq PS}
		\Vert f \Vert_{L^{\infty}(B_{1/2}(0))}\lesssim \Vert f\Vert_{L^{\infty}(E)}^{\alpha}\Vert f\Vert_{L^{\infty}(B_1(0))}^{1-\alpha},
	\end{equation}
	for a fixed set $E\subset B_{1/2}(0)$, $\alpha$ depending on $E$ and $f$ belonging in some class of functions $\mathcal F$, usually solutions of some PDE. These propagation of smallness results from $E$ to $B_{1/2}(0)$ can be interpreted as a quantitative version of unique continuation, as they imply that the only function in $\mathcal F$ that vanishes on $E$ is the zero function. 
    
	There have been a number of works studying variations of this phenomenon. Some older results that dealt with the case of $f$ being a solution to an elliptic equation with analytic coefficients and $E$ being a set of positive measure include \cite{Nad} and \cite{Vess1}, whereas \cite{nad2} and \cite{vess2} dealt with more general elliptic equations but with a worse estimate than \eqref{eq PS}. The interested reader can also look at \cite{Fos}, \cite{LM}, \cite{Mal}, \cite{zhuprop} for various results for solutions of elliptic equations with weaker hypotheses on the set $E$ than positive measure. In the parabolic setting, propagation of smallness results for harmonic functions have been useful in establishing observability inequalities, which in turn can be used to deduce null-controllability results. See for instance \cite{AEWZ}, as well as \cite{burqheat}, \cite{greenheat} for more recent applications in the parabolic setting. When $f$ solves a divergence-form second-order elliptic equation with Lipschitz coefficients, Logunov and Malinnikova showed in \cite{LM} that propagation of smallness holds for $f$ provided the smallness set $E$ has positive $(n-1+\delta)$-dimensional Hausdorff content for some $\delta>0$. Since the nodal set of $f$ can have dimension $n-1$, the extra $\delta$ in the Hausdorff content is necessary. They also showed that $|\nabla f|$ satisfies a propagation of smallness result if $E$ has positive $(n-1-c_n)$-dimensional Hausdorff content, where $c_n>0$ is a small dimensional constant. Since the critical set where $|\nabla f|=0$ has codimension 2, however, they conjectured that the correct condition on $E$ should be that it has positive $(n-2+\delta)$-dimensional Hausdorff content. Two recent works \cite{Fos}, \cite{zhuprop} have studied propagation of smallness from sets with $\delta$-dimensional Hausdorff content in the plane, but progress has been slow in higher dimensions. We aim to answer the conjecture of Logunov and Malinnikova in the affirmative for gradients of harmonic functions.
	\begin{theorem}
		\label{thm:prop_of_smallness}
		Let $m$, $\delta$ be positive numbers, $u$ be a harmonic function in the unit ball $B\subset \R^n$ such that $\sup_{B} |\nabla u| = 1$, and suppose $E \subset B_{1/2}(0)$
		satisfies $\mathcal H_\infty^{n-2+\delta}(E) > m$.
		Then there exist constants $C, \alpha >0$ depending on $m, \delta$ and $n$ only such that
		\[
		\sup_{B(0, 1/2)} |\nabla u| \leq C \left ( \sup_E |\nabla u| \right)^{\alpha}.
		\]
	\end{theorem} 
	
    We remark that Theorem \ref{thm:prop_of_smallness} was already known in the special case where $E$ lies in a codimension $1$ hyperplane, as shown in \cite{Mal}, and this case serves as the starting point of our proof. The key ingredient in our argument will be Lemma \ref{lemma:codimension_hyperplane}, which is an improved version of Logunov's hyperplane lemma for the harmonic case with scaling that is amenable to getting results in codimension $2-\delta$ for arbitrary values of $\delta>0$.  We only state and prove the result in the harmonic case, but the standard modifications to the argument should also show that it holds for operators of the form $\text{div}(A\nabla\cdot)$ where $A$ is uniformly elliptic and has analytic coefficients. It would be interesting to understand how to extend this result to elliptic operators whose coefficients are only $C^{\infty}$ or $C^k$ smooth. Theorem \ref{thm:NV}, which is taken from \cite[Theorem 1.1]{NV} implies a weaker propagation of smallness result that holds for second-order elliptic equations with Lipschitz coefficients, which we discuss in the final section of the paper, but a propagation of smallness of the form \eqref{eq PS} is not known for gradients of solutions to elliptic equations with nonanalytic coefficients when the smallness set $E$ is merely assumed to have positive $(n-2+\delta)$-Hausdorff content.

    \subsection*{Acknowledgments}
    Most of the work on this article was carried out during the second author’s visit to Stanford University. We are very grateful to Eugenia Malinnikova and to the Department of Mathematics for their hospitality, and to Eugenia Malinnikova in particular for many helpful discussions while working on this project. We are also grateful to Shaghayegh Fazliani for several helpful discussions. B.F. received partial support from NSF grant DMS-2247185 and J.G. was supported by the European Research Council (ERC) under the European Union’s Horizon 2020 research and innovation programme (grant agreement 101018680) and partially supported by MICINN (Spain) under grant PID2020-114167GB-I00, and 2021-SGR-00071 (Catalonia).
    
	\section{Some preliminaries}
	Following \cite[Section 8]{Mat95}, we define the \textit{Riesz $s$-capacity} of a set $E$ by
	\[
	\operatorname{Cap}_s(E)=\sup \left\{I_s(\mu)^{-1}: \mu(A)=1\right\},
	\]
	where the supremum is taken over all Radon measures $\mu$ compactly supported on $E$, and the $s$-energy of the measure $\mu$ is given by
	\[
	I_s(\mu)=\int|x-y|^{-s} d \mu(x) d \mu(y) .
	\]
	We also define the \textit{Hausdorff $s$-content} of a set $E$ by
	\[
	\mathcal H_\infty^s(E)=\inf \left\{\sum_{i=1}^{\infty}\left(\operatorname{diam} U_i\right)^s: \bigcup_{i=1}^{\infty} U_i \supseteq E,\right\},
	\]
	where the infimum is taken over all countable covers $(U_i)_i$ of $E$. Our arguments will rely on dividing a large cube into many subcubes and considering how many intersect a fixed sublevel set of $|\nabla u|$. If we subdivide the unit cube $Q_0$ into $K^n$ many congruent subcubes and let $S$ denote the set of cubes intersecting $E$, then we can lower bound the cardinality of $S$ via
	\[
	|S| \ge C_n \mathcal H_\infty^s(E) K^{s}.
	\]
	
	We also present a theorem from \cite[Main Result]{Mal} that will be necessary in the sequel. It gives a propagation of smallness result for sets with positive $(n-2+\delta)$-Riesz capacity, provided that it lies inside a hyperplane.
	\begin{theorem}
		\label{thm:propagation_smallness_hyperplane_Malinnikova}
		Let $u$ be a harmonic function in the unit ball $B \subset \mathbb R^n$. Suppose that $E$ is a compact subset of a codimension $1$ hyperplane, $E \subset B_{1 / 2}$, and for some $\delta>0$ the Riesz capacity $\operatorname{Cap}_{n-2+\delta}(E)=m$ is positive. Then for any compact subset $K\subset B$ and all $x\in K$
		\[
		|\nabla u(x)| \leq C \Vert \nabla u\Vert_{L^\infty(B)}^{1-\alpha} \Vert \nabla u \Vert_{L^\infty(E)}^\alpha 
		\] holds with some $C=C(n)$ and $\alpha=\alpha(K, n, \delta, m)$.
	\end{theorem}
	
    We will also be making use of triadic lattices on Euclidean space $\R^n$. Given a cube $Q$ centered at the origin of side length $\ell({Q})$ and a number $(2A+1)$ which will always be a power of 3, we partition it into $(2A+1)^n$ many congruent subcubes $q$ and call these the children of $Q$; analogously, $Q$ is the parent of each cube $q$. If we repeat the procedure by subdividing some $q$ into $(2A+1)^n$ many subcubes, we say that each child of $q$ is a descendant of $Q$; we continue to call further subcubes (obtained by iterating the subdivision procedure) descendants of $Q$, as well.
	
	\subsection*{The doubling index}
	We collect some relevant definitions and properties of the doubling index here. The doubling index is comparable to Almgren's frequency function, which has been used extensively in the study of level sets and sublevel sets of solutions to elliptic equations. The frequency function is named after Almgren, who used it in \cite{almgren} to study minimal surfaces. Using the frequency function to study elliptic PDEs goes back to the work of Garofalo and Lin \cite{GL1}, \cite{GL2}, where control of the frequency was shown to imply a doubling condition. This doubling condition could, in turn, be used to define the so-called doubling index, which was the preferred incarnation of the frequency used in the works of Logunov and Malinnikova in \cite{LogPoly}, \cite{LogLower}, \cite{LM2d3d}, \cite{LM}. For a more complete exposition on the doubling index and the frequency function, we direct the interested reader to \cite{LMLec}. Note that in our work, we will be considering doubling index for the size of gradients of solutions, as opposed to the doubling index for the size of solutions that was considered in many of these references.
	\begin{definition}[Normalized doubling index] 
		For a harmonic function $u$, we define the (normalized) doubling index of a ball $B(x,r)$ as  
		\[
		N(x,r) = \log \frac{\sup_{B(x,2r)} |\nabla u|}{\sup_{B(x,r)} |\nabla u|},
		\]
		and the (maximal normalized) doubling index of a cube $Q$ as
		\[
		 N(Q) =  \sup_{x\in Q, \,r\leq\ell(Q)}\log \frac{\sup_{B(x, 10 n r)} |\nabla u|}{\sup_{B(x,r)} |\nabla u|}
		\]
	 	where $\ell(Q)$ denotes the side length of a cube $Q$.
	\end{definition}
	Due to the equivalence of $L^p$ norms of harmonic functions, this could alternatively be formulated in terms of $L^2$ averages of $|\nabla u|$ over balls. 
	We have the following well-known almost monotonicity property of the doubling index, which goes back to the original monotonicity results and comparability of doubling index and frequency of \cite{GL1}, \cite{GL2}. This guarantees that if we instead take $r=\ell(Q)$ in the formula for $N(Q)$, we will get something comparable to the actual value of $N(Q)$.
	\begin{proposition}[Almost monotonicity of the doubling index]\label{prop almost monotonicity}
		For $\theta<1/2$, we have that there is a constant $C>0$ such that
		\[
		N(x,\theta r)\le 2N(x,r)+C
		\]
	\end{proposition}
	\begin{proof}
		 We will make use of the true monotonicity (instead of almost monotonicity) for the case where the doubling index is defined using $L^2$ averages rather than the $L^{\infty}$ norm. One reference for this that treats gradients of harmonic functions is \cite[Proposition 3]{Fos}. We denote this average by
		\[
		\langle f\rangle_{B(x,r)}=\frac{1}{\omega_n r^n}\int_{B(x,r)} |f(y)|^2\,dy,
		\]
		and we denote the alternate definition of frequency via
		\[
		N_2(x,r)=\log\frac{\langle |\nabla u|\rangle_{B(x,2r)}}{\langle |\nabla u|\rangle_{B(x,r)}}.
		\]
        The same argument that shows $N_2$ is monotone when the ratio of the radius of the balls is 2 works to show that this is monotone when replacing the ratio by any number greater than 1.
        
		It is well known that the $L^2$ and $L^{\infty}$ norms of harmonic functions are equivalent in the sense that for a harmonic function $v$, we have
		\begin{equation}\label{eq 2-infty eqce}
			\langle v\rangle_{B(x,r)} \le \Vert v \Vert_{L^{\infty}(B_r(x)} \le C_{\epsilon} \langle v\rangle_{B(x,(1+\epsilon)r)}.
		\end{equation}
		We can also get a version of this for gradients since each partial derivative $\partial_{x_i}v$ is a harmonic function. Now, if $\theta\le 1/2$, we have that
		\begin{align*}
			N(x,\theta r)=\log\frac{\Vert \nabla u\Vert_{L^{\infty}(B(x,2\theta r))}}{\Vert \nabla u\Vert_{L^{\infty}(B(x,\theta r))}} &\le \log \frac{\langle |\nabla u|\rangle_{B(x,(2+\epsilon)\theta r)}C_{\epsilon}}{\langle |\nabla u|\rangle_{B(x,\theta r)}} \\
			&\le \tilde{C}_{\epsilon} +\log\frac{\langle |\nabla u|\rangle_{B(x,2r)}}{\langle |\nabla u|\rangle_{B(x,2r/(2-\epsilon))}}\frac{\langle |\nabla u|\rangle_{B(x,2r/(2-\epsilon))}}{\langle |\nabla u|\rangle_{B(x,2r/(2+\epsilon))}},
        \end{align*}
        where we used monotonicity of $N_2$ with a ratio of $2+\epsilon$. Note that we have the estimates
        \[
        \frac{\langle |\nabla u|\rangle_{B(x,2r/(2-\epsilon))}}{\langle |\nabla u|\rangle_{B(x,2r/(2+\epsilon))}}\le \frac{\langle |\nabla u|\rangle_{B(x,2(2+\epsilon)r/(2-\epsilon)^2)}}{\langle |\nabla u|\rangle_{B(x,2r/(2-\epsilon))}}\le\frac{\langle |\nabla u|\rangle_{B(x,2r)}}{\langle |\nabla u|\rangle_{B(x,2r/(2-\epsilon))}}
        \]
        due to monotonicity of $N_2$ as well as monotonicity of $L^2$ averages of $|\nabla u|$. Thus, we deduce
            \[
			N(x,\theta r)\le \tilde{C}_{\epsilon}+2N(x,r),
		\]
		where we have made use of \eqref{eq 2-infty eqce}. Taking $\epsilon=1/100$, for instance, completes the proof.
	\end{proof}
    \begin{remark}
        The entire proof still works with a different ratio $L$ in the definition of doubling index, provided that $\theta<1/L$ is imposed.
    \end{remark}
    As a corollary, we see that we can take $r=\ell(Q)$ in the definition of maximal doubling at the loss of a constant.
    \begin{corollary}[Comparability of maximal doubling and doubling]
        \[
        N(Q)\le C_0+C \sup_{x\in Q} \log \frac{\sup_{B(x, 10 n \ell(Q))} |\nabla u|}{\sup_{B(x,\ell(Q))} |\nabla u|}
        \]
    \end{corollary}
    \begin{proof}
        First, consider the case when the value of $r$ achieving the supremum satisfies $r\le \ell(q)/10n$. Then we have that
        \[
        N(Q)\le \tilde{C}_{\epsilon}+2\sup_{x\in Q} \log \frac{\sup_{B(x, 10 n \ell(Q))} |\nabla u|}{\sup_{B(x,\ell(Q))} |\nabla u|}.
        \]
        Now, suppose instead that $r>\ell(q)/10n$. Then we have that
        \[
        N(Q)\le \sup_{x\in Q} \log \frac{\sup_{B(x, 10 n \ell(Q))} |\nabla u|}{\sup_{B(x,\ell(Q)/10n)}|\nabla u|}\le C'+3\sup_{x\in Q} \log \frac{\sup_{B(x, 10 n \ell(Q))} |\nabla u|}{\sup_{B(x,\ell(Q))} |\nabla u|}.
        \]
        where we used monotonicity to control the quotient
        \[
        \log \frac{\sup_{B(x, \ell(Q))} |\nabla u|}{\sup_{B(x,\ell(Q)/10n)}|\nabla u|}.
        \]
    \end{proof}
A further corollary to the monotonicity allows us to estimate ratios of maximum values over vastly different scales when we know the frequency on the smaller scale.
\begin{corollary}\label{cor big scale doubling}
    Suppose $N(x,tr)\ge N>N_0$ where $N_0$ is universal and $t<1/2$. Then
    \[
    \sup_{B(x,tr)}|\nabla u| \le Ct^{N/6}\sup_{B(x,r)} |\nabla u|.
    \]
\end{corollary}
\begin{proof}
    Assume first $t=2^{-j}$ for some natural number $j$. By Proposition \ref{prop almost monotonicity}, we have for any $i\ge1$ that
    \[
    N(x,2^itr)\ge N(x,tr)/2-c'\ge N/3,
    \]
    provided that $N_0$ is chosen to be large enough. It follows that
    \[
    \log\frac{\sup\limits_{B(x,r)}|\nabla u|}{\sup\limits_{B(x,2^{-j}r)}|\nabla u|}=\sum_{i=1}^j N(x,2^{-i}r)\ge jN/3.
    \]
    Rearranging gives the claimed inequality. For more general values of $t$, with $2^{-j}<t<2^{-(j+1)}$, we use the simple inequality
    \[
    \frac{\sup\limits_{B(x,r)}|\nabla u|}{\sup\limits_{B(x,r/2)}|\nabla u|}\le \frac{\sup\limits_{B(x,r)}|\nabla u|}{\sup\limits_{B(x,2^{j-1}tr)}|\nabla u|}
    \]
    and repeat the argument by lower bounding
    \[
    \log\frac{\sup\limits_{B(x,r)}|\nabla u|}{\sup\limits_{B(x,tr)}|\nabla u|}\ge N(x,r/2)+\sum_{i=0}^{j-2} N(x,2^{i}tr)\ge jN/3.
    \]
\end{proof}
    
	The general strategy is to show that the lower bound on the Hausdorff content of the set $\{|\nabla u|<\epsilon\}$ implies that the doubling index is sufficiently large, which in turn implies a propagation of smallness result. This second implication is based on the observation that if we have the normalization $\Vert\nabla u\Vert_{L^{\infty}(B_1)}=1$ and $N(0,1/2)>\gamma\log(1/\epsilon)$ then
	\begin{equation}\label{eq final implication}
	\epsilon^{-\gamma} < \frac{\sup_{B_1(0)} |\nabla u|}{\sup_{B_{1/2}(0)} |\nabla u|}=\frac{1}{\sup_{B_{1/2}(0)}|\nabla u|},
	\end{equation}
	and, hence, $\sup_{B_{1/2}(0)}|\nabla u|<\epsilon^{\gamma}$. 
	
	\section{Proof of Theorem \ref{thm:prop_of_smallness}}
	\subsection{Codimension $2-\delta$ hyperplane lemma}
	The original hyperplane lemma of \cite{LogPoly} shows a semiadditivity property for the doubling index over cubes distributed over a codimension $1$ hyperplane. Using Theorem \ref{thm:propagation_smallness_hyperplane_Malinnikova} we can prove a similar semiadditivity result for gradients of harmonic functions, but in the setting where the cubes are covering a set with positive Riesz $(n-2+\delta)$ capacity.
	
	In the sequel, let $Q$ be the unit cube in $\R^n$ and assume $u$ is a harmonic function defined over $100Q \subset \mathbb R^n$. Throughout this section, we will assume $(2A+1)$ is a power of $3$  and we will say that a cube $q$ descended from $Q$ is \textit{bad} if its doubling index $N(q)$ is larger than a fixed constant $N$. We will think of $N$ as being half of the frequency of $\nabla u$ over $Q$. 
	\begin{lemma}[Hyperplane lemma with Riesz capacities]
		\label{lemma:hyperplane_riesz}
		Let $m >0$, $\delta \in (0,1)$, and $A$ be a positive integer. Divide $Q$ into $(2 A+1)^{n}$ equal subcubes $q_{i}$ with side-length $\frac{1}{2 A+1}$ and denote by $q_{i, 0}$ the cubes that have nonempty intersection with the hyperplane $\{x_{n}=0\}$. 
		There exist $A_{0} = A_0(m, \delta, n)$ and $N_{0} = N_0(n)$ such that if $A>A_{0}$ and $N>N_{0}$ and
		\[
		\operatorname{Cap}_{n-2+\delta}\left(\{x_n=0\}\cap \bigcup_{N(q_{i,0}) > N} q_{i,0}\right) \geq m
		\]
        then we have that $N(2Q)>2 N$.
	\end{lemma}
	
	The proof follows the original proof of Logunov \cite[Lemma 4.1]{LogPoly} but with the obvious modification of using Theorem \ref{thm:propagation_smallness_hyperplane_Malinnikova} instead of propagation of smallness from a hyperplane. 
	\begin{proof}
		Let $M = \sup_{2Q} |\nabla u|$, and let $x_i$ be a point in $q_{i,0}$ where the supremum of $N(x_i,\ell(q_{i,0}))$ is achieved. For any bad cube $q_{i,0}$ {(cubes with $N(q_{i,0})>N$)}, we have
		\[
		\sup_{2 q_{i,0}} |\nabla u| \leq \sup_{B(x_i, 2\sqrt n/(2A+1) )} |\nabla u| \leq \sup_{B(x_i, 1/(20n))} |\nabla u| \left (  \frac{{40 n^{3/2}}}{2A +1}\right)^{N/C-C'} 
		\leq M 2^{- C'' N \log A},
		\]
		where we used Corollary \ref{cor big scale doubling} in the second inequality. Using Theorem \ref{thm:propagation_smallness_hyperplane_Malinnikova} with 
		\[E := \{x_n=0\}\cap\left(\bigcup\limits_{q_{i,0} \text{ bad}} q_{i,0}\right),
		\]
		we can upper bound $|\nabla u(x)|$ on $Q$ by $M 2^{-C''N\alpha \log A}$ for some $\alpha>0$ independent of $A$. Then, if $A$ is large enough depending on $\alpha$, we get that $N(2Q)$ must be larger than $2N$.
		
	\end{proof}
	The previous lemma gives us information on the Riesz capacity of the cubes with large doubling index. We will use this to show that if $N(2Q)\leq 2N$, then the set of bad cubes {has to be} small in the sense that they make up a small fraction of all the cubes. 
	
	\begin{lemma}[Hyperplane lemma with codimension $(2-\delta)$ Hausdorff measure]
		\label{lemma:codimension_hyperplane}
		Assume $N(2Q) \leq 2N$, and fix $\delta>0$. Divide $Q$ into $(2 A+1)^{n}$ equal subcubes $q_{i}$ with side length $\frac{1}{2 A+1}$ and denote by $q_{i, 0}$ the cubes that have nonempty intersection with the hyperplane $\{x_{n}=0\}$. 
		Then, there exist $A_0(\delta,n, \eta)$ and $N_0(n)$ such that if $A > A_0$, and $N>N_{0}$, the number of cubes $q_{i,0}$ satisfying $N(q_{i,0})> N$ is smaller than  $\eta(2A+1)^{n-2+\delta}$.
	\end{lemma}
	
	Note that the original hyperplane lemma \cite[Corollary 4.2]{LogPoly} is precisely this lemma with $\delta=1$, and also holds for more general elliptic operators with Lipschitz coefficients. For our applications, it will suffice to take $\eta$ to be a small constant.
	\begin{proof}
		Fix $\delta>0$, $N_0$ large enough so that Lemma \ref{lemma:hyperplane_riesz} can be applied, and fix $A$ large enough, and assume for the sake of contradiction that $u$ is a harmonic function satisfying the hypotheses of the lemma with too many bad cubes $q_{i,0}$ satisfying $N({q_{i,0}})>N$. We will prove two intermediate claims.
		
		\begin{claim}[Claim 1] \label{claim:1}
			Fix $C_n>0$ and let $\mu$ be a measure of the form 
			\[
			\mu=\frac{\mathcal{H}^{n-1}|_E}{\mathcal{H}^{n-1}(E)}
			\]
			where $E\subset\{x_n=0\}$ is a nonempty finite union of lattice cubes of side length $(2A+1)^{-1}$, all contained in the unit cube.  If $\operatorname{Cap}_{n-2+\delta/4}(E)$  is small enough depending on $C_n$  (but independent of $A$), then there exists a ball $\wt B$ of the hyperplane $\{x_n = 0\}$
			such that $\mu(\wt B) > C_nr(\wt B)^{n-2+\delta/2}$. 
		\end{claim}
		 In the particular case $E$ is the union of bad cubes with side length $(2A+1)^{-1}$ intersected with the hyperplane $\{x_n=0\}$, Claim $1$ implies that there exists $A_0(n,\delta)$ such that if $A>A_0$, then we can find a ball $\wt B$ of the hyperplane satisfying $\mathcal H^{n-1}(\wt B)/\mathcal H^{n-1}(E) > C_n r(\wt B)^{n-2+\delta/2}$. This is because Lemma \ref{lemma:hyperplane_riesz} implies that $\operatorname{Cap}_{n-2+\delta/4}(E)$ can be taken to be arbitrarily small by choosing $A$ sufficiently large.
		\begin{proof}[Proof of Claim 1]
			
			Assume that for any ball $\wt B$ of the hyperplane $\{x_n=0\}$, we have $\mu(\wt B) \leq C_n r(\wt B)^{n-2+\delta/2}$.
			Then, for fixed $x\in E$, we get
			\begin{align*}
				\int_{E} |x-y|^{-(n-2+\delta/4)}\, d\mu(y)
				&= \mathcal{H}^{n-1}(E)^{-1} \int_{0}^{\sqrt{n-1}} \left( \int_{\partial B_r(x) \cap \{x_n=0\}}\mathbf{1}_E(y)r^{-(n-2+\delta/4)}\,d\mathcal{H}^{n-2}(y)  \right) \,dr \\
				&= \int_{0}^{\sqrt{n-1}} r^{-(n-2+\delta/4)}\left( \mathcal{H}^{n-1}(E)^{-1} \int_{\partial B_r(x) \cap \{x_n=0\}}\mathbf{1}_E(y)\,d\mathcal{H}^{n-2}(y)  \right) \,dr \\
				&= \int_0^{\sqrt{n-1}} r^{-(n-2+\delta/4)} \frac d {dr} \mu(\{y : |y-x|\leq r\}) \,dr
				\\
				&= C(n)\mu(E) - \int_0^{\sqrt{n-1}} \frac d {dr} r^{-(n-2+\delta/4)} \mu(\{y : |y-x|\leq r\}) \,dr\\
				&=C(n)+ (n-2+\delta/4) \int_0^{\sqrt{n-1}} r^{-(n-2+\delta/4)-1} \mu(B(x,r)) \, dr
			\end{align*}
			where we have used that $\lim_{r\to0} \mu(B(x,r))r^{-(n-2+\delta/4)}=0$ as $\mu$ is ($n-1$)-dimensional at scales below $(2A+1)^{-1}$. We will ignore the additive constant $C(n)$ since it is finite.
			We integrate the above identity in $x$ with respect to $d\mu$ using Fubini's Theorem
			\[
			C_{n,\delta}\int_E\int_0^{\sqrt{n-1}} r^{-(n-2+\delta/4)-1} \mu(B(x,r)) \, dr \, d\mu(x)
			=
			C_{n,\delta}\int_0^{\sqrt{n-1}} r^{-(n-2+\delta/4)-1} \int_E    \mu(B(x,r)) \, d\mu(x)\, dr 
			\]
			and finally, we bound
			\[
			\int_E \mu(B(x,r)) \,d\mu(x) \leq c_nr^{n-2+\delta/2} \mu(E) = C_nr^{n-2+\delta/2}.
			\]
			Thus, we obtain
			
			\[
			I_{n-2+\delta/4}(\mu) \le C_n \int_0^{\sqrt{n-1}} r^{\delta/4-1}\,dr =C_n C(n,\delta)<+\infty
			\]
			as we get a convergent integral
			and
			\[
			\operatorname{Cap}_{n-2+\delta/4}(E) \geq I_{n-2+\delta/4}(\mu)^{-1}\geq C_n^{-1}C(n,\delta)^{-1}>0.
			\]
		\end{proof}
		\begin{claim}[Claim 2]
			Fix $K>0$. There exists $A_0(K,\delta, \eta )$ such that if $A>A_0$, then any triadic cube $\wt Q$ in the hyperplane $\{x_n=0\}$ satisfying $\mu(\wt Q) > \ell(\wt Q)^{n-2+\delta/2}$ must also satisfy $\ell(\wt Q)> K(2A+1)^{-1}$.
		\end{claim}
		\begin{proof}[Proof of Claim 2]
			By definition of $\mu$ and the assumption that the statement of Lemma \ref{lemma:codimension_hyperplane} fails (that is, $\mathcal H^{n-1}(E)\geq \eta (2A+1)^{n-2+\delta}/(2A+1)^{n-1}$), we have
			\[
			\mu(\wt Q) = \frac{\mathcal H^{n-1}(\wt Q \cap E)}{\mathcal H^{n-1}(E)}\leq  \frac{\mathcal H^{n-1}(\wt Q )}{\mathcal H^{n-1}(E)} 
			= \frac{ \ell(\wt Q)^{n-1}}{\mathcal H^{n-1}(E)} \leq 
			\frac{ \ell(\wt Q)^{n-1}}{\eta (2A+1)^{n-2+\delta}/ (2A+1)^{n-1}} = 
			\frac{ \ell(\wt Q)^{n-1}}{\eta (2A+1)^{-1+\delta}} .
			\]
			Hence,
			\[
			\ell(\wt Q)^{n-2+\delta/2} <  \frac{ \ell(\wt Q)^{n-1}}{\eta (2A+1)^{-1+\delta}} \quad \implies 
			(2A+1)\ell(\wt Q) > c(\delta,\eta ) (2A+1)^{1-\frac{1-\delta}{1-\delta/2}}.
			\]
			Choosing $A_0$ sufficiently large so that $c(\delta,\eta )(2A+1)^{1-\frac{1-\delta}{1-\delta/2}}>K$ gives the claim. 
		\end{proof}
		With these established, we can complete the proof of Lemma \ref{lemma:codimension_hyperplane}. Let 
		\[
		E:= \left(\bigcup_{N(q_{i,0})>N} q_{i,0} \right) \cap  \{x_n = 0\}, \quad \mbox{and}\quad \mu :=\frac{\mathcal H^{n-1}|_{E}}{\mathcal H^{n-1}(E)}.
		\]
        Let $\wt P$ be a minimal triadic cube in the hyperplane $\{x_n=0\}$ satisfying $\mu(\wt P) > \ell(\wt P)^{n-2+\delta/2}$. Although it may not be unique, we know there exists at least one of these cubes by Claim 1 and, by Claim 2, the side length of $\tilde{P}$ is at least $(2A+1)^{-1}$.
		
		Let $P$ denote the triadic cube in $\R^n$ with the same side length as $\wt P$ and intersecting it; that is, $P=\wt P\times [-\ell(P)/2,\ell(P)/2]$. Consider the original function $u$ restricted to $P$, and note $N<N(P)\leq2N$ by the definition of $E$.

		Let $K$ be such that $\ell(P)=K(2A+1)^{-1}$. Clearly, $K$ is a power of $3$ and, by Claim $2$, we can ensure that it is as large as we want by making $A$ larger. Define the measure $\nu := \mu(\wt P)^{-1}\cdot \mu|_{\wt P}$. Now, by minimality of $\wt P$, this measure satisfies for every triadic descendant cube $\wt R \subset \wt P$ that
		\begin{equation}\label{nu bound}
			\nu(\wt R) < \left(\frac{\ell(\wt R)}{\ell(\wt P)}\right)^{n-2+\delta/2}.
		\end{equation}
		We can rescale $\wt P$ so it becomes the unit cube. The original division of $\wt Q$ into cubes of side length $(2A+1)^{-1}$ induces a subdivision of $\wt P$ into cubes of side length $K^{-1}$. However, notice that $\nu$ satisfies all the hypotheses of Claim 1 by construction. Consider Claim 1 and take a descendant cube contained inside it. If $C_n$ is large enough in Claim 1, then combining this with \eqref{nu bound} leads to a contradiction  if we choose $K$ large enough, as permitted by Claim 2.
	\end{proof}

	\begin{remark}
		\label{rmk take arbitrary C}
		In Lemma \ref{lemma:hyperplane_riesz}, it is easy to check that, for any $C>1$, we can also ensure that $N(2Q) \geq CN$ by taking $A_0 = A_0(m,\delta,n,C)$ large enough.
		In a similar vein, in Lemma \ref{lemma:codimension_hyperplane}, we can change the assumption  $N(2Q)\leq 2N$ for $N(2Q)\leq CN$ by increasing $A_0$ accordingly.
	\end{remark}
	
	\subsection{Number of cubes with large doubling index}
	Now, we can prove a theorem analogous to \cite[Theorem 5.1]{LogPoly} showing that the number of bad cubes (inside the whole cube, not only the hyperplane) is also small. Here, bad cubes are cubes with frequency at least $N$, and we assume the frequency on the largest cube $Q$ is at most $(1+c)N$ where $c$ is a suitably small constant.
	\begin{theorem}[Bounding the count of bad subcubes]
		\label{thm:cubes_with_large_index}
		Fix $\eta>0$. There exists a constant $c>0$ such that for all $\delta>0$, $A\geq A_0(\delta,n)$, and $N>N_0(n)$ the following statement holds. If we partition a cube $Q$  satisfying $N(Q)\leq (1+c)N$ into $(2A+1)^{n}$ equal subcubes, then the number of subcubes with doubling index greater than $N$ is less than $\frac 1 2 (2A+1)^{n-2+\delta}$. 
	\end{theorem}
	The proof is analogous to \cite[Theorem 5.1]{LogPoly} but using the improved hyperplane lemma (see also \cite[Section 5]{LM}). For the reader's convenience, we present the details here as well. Fix a small $\eta$ and $A$ accordingly for which Lemma \ref{lemma:codimension_hyperplane} holds. We follow the same setup: given a cube $Q$, partition it into $(2A+1)^n$ many subcubes and iterate the process across $k$ generations. Fix $c>0$ to be determined later, consider a $j$th generation subcube $q$, that is, a cube with $\ell(q) = (2A+1)^{-j}$, and define
	\[
	F =\left\{x\in q: \sup_{r<\diam(q)/(2A+1)} N(x,r)>N(Q)/(1+c)\right\},
	\]
	so that any bad child cube of $q$ has nonempty intersection with $F$ as $N(Q)\leq (1+c)N$.
	By the width of {$F$}, we mean the minimum distance between the boundaries of two half spaces with parallel boundaries $H,H'$ such that $F\subset H\cap H'$.
	To control it, we will make use of the following geometric lemma, which was \cite[Lemma 5.3]{LogPoly}. Although his version was in terms of doubling index based on $|u|$ rather than $|\nabla u|$, the argument is unchanged due to almost monotonicity of the frequency holding in both settings. 
	\begin{lemma}[Thinness of the bad set]\label{LogGeomLem}
		With the setup as above, for any $w_0>0$, there exist $j_0,c_0>0$ such that if $j>j_0$ and $c<c_0$ then $\text{width}(F)<w_0\diam(q)$.
	\end{lemma}
	As in \cite{LogPoly}, we {continue} with a slightly easier lemma before proving Theorem \ref{thm:cubes_with_large_index}. This lemma differs from Theorem \ref{thm:cubes_with_large_index} because it only bounds the number of bad children for cubes obtained after sufficiently many subdivisions of the original cube.
	\begin{lemma}[Bound on the number of bad children cubes of descendant cubes]
		\label{lem bound number of bad children}
		Suppose {$\eta,c$} are sufficiently small and let $j_0=j_0(\eta,c)$ be large enough.
		Then, if $j\geq j_0$ in the definition of $q$, at most $\frac 1 2(2A+1)^{n-2+\delta}$ many of the children of $q$ satisfy $N(q) > N$.
	\end{lemma}
	\begin{proof}[Proof of Lemma \ref{lem bound number of bad children}]
		Take $w_0=(18A+9)^{-1}$ in Lemma \ref{LogGeomLem}, which gives some $j_0,c_0>0$. We have that $F\subset B_{w_0\diam(q)}(P)$ for some hyperplane $P$ by Lemma \ref{LogGeomLem} and after applying a rotation, we can assume $P$ is parallel to some coordinate hyperplane. 
        
        Let $\hat q$ be the smallest cube with center in $P$ and with its sides aligned with $P$ such that $q\subset \hat q$ (note that this cube may not be unique). From the fact that $w_0$ is small, we have that $P$ intersects $2q$ and hence $\ell(\hat q) \lesssim \ell(q)$. From now on, we will also consider the triadic structure generated by the cube $\hat q$.	If we divide $\hat q$ into $(2A+1)^{n}$ congruent subcubes, then the number of bad subcubes of $\hat q$ that intersect $P$ controls the total number of bad subcubes of $q$ (up to some dimensional multiplicative factor) as $F \subset B_{w_0 \diam(q)}(P)$ which in turn is contained in the union of all children of $\hat q$ that intersect $P$. 
		
		Let's control $N(\hat q)$. If $\hat q \subset Q$, then we have $N(\hat q)\leq N(Q)$ by definition of (maximal) doubling index. If $\hat q \not\subset Q $ (which can be the case if $q$ is very close to $\partial Q$), Proposition \ref{prop almost monotonicity} (almost monotonicity of $N$) together with the fact that $\hat q \cap Q \neq \emptyset$, $\ell(\hat q) \ll \ell(Q)$, and $N(\hat q)>N_0 \gg 1$ implies that $N(\hat q) \leq C N(Q)$ for some universal constant $C$. In either case, we are now in the setting of Lemma \ref{lemma:codimension_hyperplane} (together with Remark \ref{rmk take arbitrary C} and with $P$ playing the role of $\{x_n=0\}$) and it implies that $\hat q$ has at most $\eta(2A+1)^{n-2+\delta}$ bad children subcubes intersecting $P$. Hence, $q$ has at most $C_n \eta(2A+1)^{n-2+\delta}$ bad children in total for some dimensional constant $C_n>0$, so by choosing $\eta = (2C_n)^{-1}$ we obtain the desired result.
		\end{proof}
	The proof of Theorem \ref{thm:cubes_with_large_index} now follows by subdividing $Q$ into $(2A+1)^n$ congruent subcubes and iterating the process enough times. Indeed, let $L_j$ denote the count of bad cubes {of generation $j$} of this subdivision procedure. Any good cube has the property that all its descendants are good by monotonicity, so the upper bound on the number of bad children of a fixed cube implies that
	\[
	L_{j+1}\le \frac{1}{2} (2A+1)^{n-2+\delta} L_j
	\]
	Thus,
	\[
	L_{j_0+k}\le \left (\frac{1}{2} (2A+1)^{n-2+\delta}\right)^{k}L_{j_0} \leq \frac {(2A+1)^{(2-\delta)j_0}} {2^{k}} (2A+1)^{
	(n-2+\delta)(k+j_0)},
	\]
	 which is smaller than $\frac 1 2 (2A+1)^{
	(n-2+\delta)(k+j_0)}$ if $k$ is large enough to beat the contribution from $L_{j_0}$.

    \subsection{Estimates for the sublevel set in the case of bounded frequency}
	With Theorem \ref{thm:cubes_with_large_index} established, Theorem \ref{thm:prop_of_smallness} follows via an analogous argument to the one in \cite[Section 5]{LM}. We recall that the base case {$N(Q)\leq N_0$ is a consequence of a} volume estimate on the ``effective critical set" of the solution. We include this estimate, \cite[Theorem 1.1]{NV}, for convenience.
	First, we introduce the \textit{$r$-effective critical set }of a harmonic function $u$ by
	\[
	\mathcal{C}_r(u)=\left\{x:\inf_{B_r(x)}|\nabla u|^2 <\frac n {16} r^{-2}\fint_{\partial B_{2r}(x)} |u(y)-u(x)|^2\,dy\right\}.
	\] 
	\begin{theorem}\label{thm:NV}
		Let $u:B_2\rightarrow \mathbb R^n$ be a solution to the PDE
		\[
		\operatorname{div}A\nabla u+\mathbf{b}\cdot \nabla u=0.
		\]
		If the doubling index of $\nabla u$ is bounded above on $B_1(0)$ by {$N$}, then for all $0<r<1$
		\[
		\text{Vol}\,(B_{1/2}(0)\cap B_r(\mathcal{C}_r(u))) \le C^{N^2}r^2
		\]
		for all $N\ge 1$.
	\end{theorem}
	In particular, their proof shows that the $r$-effective critical set can be covered by at most $C^{N^2}$ many balls of radius $r$. As a result, if the unit cube is subdivided into $K^n$ many congruent subcubes, taking $r=K^{-1}$, implies that {there exists $c>0$ such that} the number of cubes $q$ where
	\[
	\inf_q |\nabla u|< c\sup_{2q} |\nabla u|
	\]
	is at most $C^{N^2}K^{n-2}$.
	With this, we will establish the base case, which follows analogously to the argument in \cite{LM}.
	\begin{lemma}\label{lem base case}
		Suppose that $u:Q\rightarrow\R$ where $Q\subset \R^n$ is the unit cube is a harmonic function satisfying $\Vert \nabla u\Vert_{L^{\infty}(Q)}=1$ and $N(Q)\le  N_0$. Then for every $\delta>0$, there are constants $C(\delta),\beta(\delta)>0$ such that
		\[
		\mathcal{H}_{\infty}^{n-2+\delta}(|\nabla u|^{-1}([0,e^{-a}]))<C(\delta)^{N_0^2}e^{-\beta(\delta)a/N_0}
		\]
		holds for all $a>0$.
	\end{lemma}
	\begin{proof}
		Divide $Q$ into $K^n$ (where $K$ will be determined later) many congruent cubes, and notice that on any child cube $q$  far from the effective critical set, we have the lower bound
		\[
		\inf_{q} |\nabla u| \ge c \sup_{2q} |\nabla u| \ge cK^{-CN_0}
		\]
		due to monotonicity of frequency. Thus, it follows that if we take $K=e^{\gamma a/N_0}$ for sufficiently large $\gamma$, then we have that $|\nabla u|^{-1}([0,e^{-a}])$ is contained in the union of at most $C^{N_0^2}K^{n-2}$ many bad cubes. This immediately gives the desired bound for the $(n-2+\delta)$-Hausdorff content of the sublevel set.
	\end{proof}
	Thus, the base case in \cite{LM} still works if the smallness set has positive $(n-2+\delta)$-Hausdorff content for any $\delta>0$, but their induction step (following from \cite[Lemma B]{LM}) only works if $\delta>1-c_n$. However, we can replace their usage of \cite[Lemma B]{LM} by our own Theorem \ref{thm:cubes_with_large_index}; this allows the induction step to work for any $\delta>0$, completing the proof. 
    \subsection{Recursive inequality}
    For the reader's convenience, we present the details for how to derive the relevant recursive inequality.
    \begin{proposition}
        Given a harmonic function $u$ and parameters $a>0$ and $\delta>0$, we let
        \[
        m(u,a)=\mathcal{H}_{\infty}^{n-2+\delta}(\{x\in Q_0:|\nabla u(x)|<e^{-a}\sup_{Q_0} |\nabla u|\})
        \]
        Given $N>0$, we define $M(N,a)$ to be the maximum of $m(u,a)$ across all harmonic functions with frequency on $Q_0$ bounded above by $N$. Then
        \[
        M(N,a)\le A^{2-\delta}M(N/(1+c),a-C_1N\log A)+A^{-\delta}M(N,a-C_1\log N)
        \]
    \end{proposition}
    \begin{proof}
        Fix an arbitrary such solution $u$. Let $A$ be large enough, as in Theorem \ref{thm:cubes_with_large_index} with parameter $\delta/2$, and apply the theorem to partition $Q_0$ into $A^n$ many congruent subcubes $\mathcal{Q}=\{q_j\}_{j=1}^{A^n}$. Subadditivity of Hausdorff content implies that
        \begin{align*}
        m(u,a)&\le \sum_{j=1}^{A^n}\mathcal{H}_{\infty}^{n-2+\delta}(\{x\in q_j:|\nabla u(x)|\le e^{-a}\sup_{Q_0}|\nabla u|\})\\
        &\le  \sum_{j=1}^{A^n}\mathcal{H}_{\infty}^{n-2+\delta}(\{x\in q_j:|\nabla u(x)|\le e^{-a}B^{C_1N}\sup_{q_j}|\nabla u|\}),
        \end{align*}
        where we used the upper bound on the doubling. Now, we split $\mathcal{Q}=\mathcal{G}\cup \mathcal{B}$ into good and bad cubes, where a cube $q$ is bad if $N(q)>N/(1+c)$. Theorem \ref{thm:cubes_with_large_index} lets us easily bound the contribution from bad cubes as
        \[
        \sum_{q\in\mathcal{B}}\mathcal{H}_{\infty}^{n-2+\delta}(\{x\in q:|\nabla u(x)|\le e^{-a}A^{C_1N}\sup_{q}|\nabla u|\})\le A^{n-2+\delta/2}A^{-(n-2+\delta)}M(N,a-C_1N\log A).
        \]
        Here, we made use of the invariance of frequency under rescaling $q_j$ back to being unit size. Similarly, we can bound the good cubes as
        \[
        \sum_{q\in\mathcal{G}}\mathcal{H}_{\infty}^{n-2+\delta}(\{x\in q:|\nabla u(x)|\le e^{-a}A^{C_1N}\sup_{q}|\nabla u|\})\le A^{n-2+\delta/2}A^{-n}M(N/(1+c),a-C_1N\log A).
        \]
        Combining these gives the desired recursive inequality.
    \end{proof}
    We will now repeat the argument of Logunov and Malinnikova to show that this gives the desired propagation of smallness, i.e. Theorem \ref{thm:prop_of_smallness}. It will suffice to show that
    \begin{equation}\label{eq exponential bound}
    M(N,a)\le Ce^{-\beta a/N}.
    \end{equation}
    If we know \eqref{eq exponential bound}, then using the lower bound on Hausdorff content, we see that $N>ca$ holds necessarily, and then the claim follows by \eqref{eq final implication}. We will now prove \eqref{eq exponential bound} by double induction in $a$ and $N$; the base case was already discussed. Suppose we know \eqref{eq exponential bound} for some $N/(1+c)$ and all $a>0$ as well as for $N$ and $a\le a_0$. We will show the inequality holds for $N$ and $a_0+C_1N\log A$. We may assume $a_0>C'N\log A$ as otherwise the result follows from compact sets having finite Hausdorff content. The recursive inequality gives
    \[
    M(N,a_0+C_1N\log A)\le A^{2-\delta} Ce^{-\beta a_0(1+c)/N}+A^{-\delta}Ce^{-\beta a_0/N} \stackrel{?}{\le} Ce^{-{\beta (a_0+C_1N\log A)/N}}
    \]
    Dividing out by common factors, this reduces to showing
    \[
    A^{2-\delta} e^{-c\beta a_0/N}+A^{-\delta} \stackrel{?}{\le} A^{-C_1\beta}
    \]
    Since $a_0/N>C'\log A$, we get the new inequality
    \[
    A^{2-\delta} A^{-cC'\beta}+A^{-\delta} \stackrel{?}{\le} A^{-C_1\beta}.
    \]
    We are free to choose $\beta>0$ small and $C'>0$ large so that the desired inequality holds. This completes the proof.
	\section{A weaker propagation of smallness for nonanalytic coefficients}
	For the reader's convenience, we state here a weaker propagation of smallness result for elliptic operators with Lipschitz coefficients, which follows as a corollary to Lemma \ref{lem base case}. Note that although Lemma \ref{lem base case} is stated for harmonic functions, its proof follows by applying Theorem \ref{thm:NV}, which is valid for solutions to elliptic equations with Lipschitz coefficients, and Lemma \ref{lem base case} also holds in this generality.

    \begin{corollary}
        Let $m$, $\delta$ be positive numbers and let $L$ be an elliptic operator of the form
        \[
        Lf=\nabla\cdot(A\nabla f)+\mathbf{b}\cdot\nabla f
        \]
        where $A$ is uniformly elliptic with Lipschitz coefficients and $\mathbf{b}$ is bounded and measurable. Let $u$ be a solution to $Lu=0$ in the unit ball $B\subset \R^n$ such that $\sup_{B} |\nabla u| = 1$, and suppose $E \subset B_{1/2}(0)$
		satisfies $\mathcal H_\infty^{n-2+\delta}(E) > m$ and denote $\epsilon=\sup_{E}|\nabla u|$.
		Then there exist constants $C, \alpha >0$ depending on $m, \delta$ and $n$ only such that
		\[
		\sup_{B(0, 1/2)} |\nabla u| \leq \exp\left(-C(\log(1/\epsilon))^{1/3}\right)
		\]
    \end{corollary}
    \begin{proof}
    Denote $\kappa=\mathcal{H}_{\infty}^{n-2+\delta}(|\nabla u|^{-1}([0,e^{-a}]))>0$ where $\epsilon=e^{-a}$. Then by Lemma \ref{lem base case}, we have
	\[
	\kappa\le C^{N^2}\epsilon^{\beta /N}.
	\]
	Rearranging gives
	\[
	N^3\log C -N \log \kappa \ge \beta \log(1/\epsilon).
	\]
	Hence, we deduce that
	\begin{equation}\label{eq weak GPS}
		N \ge C(\beta,\kappa)\log(1/\epsilon)^{1/3}.
	\end{equation}
    \end{proof}
	This is weaker than the ideal propagation of smallness in which the $\log(1/\epsilon)$ term has power 1. Improvements in the dependence on $N$ in the bound of Theorem \ref{thm:NV} would theoretically improve the bound in \eqref{eq weak GPS}, but even a hypothetical quadratic dependence on $N$ would not be enough to deduce optimal propagation of smallness in this manner.
	
	\bibliographystyle{alpha}
	\bibliography{sources}

\end{document}